\documentclass[reqno]{amsart}

\usepackage{amssymb,amsmath,amsthm,xcolor,enumerate,comment,enumitem,bm,longtable,hyperref,cleveref}
\usepackage[normalem]{ulem}
\usepackage[utf8]{inputenc}

\allowdisplaybreaks

\newtheorem{thrm}{Theorem}[section]
\newtheorem{cor}[thrm]{Corollary}
\newtheorem{lem}[thrm]{Lemma}
\newtheorem{prop}[thrm]{Proposition}

\theoremstyle{definition}
\newtheorem{defn}[thrm]{Definition}

\newtheorem{rem}[thrm]{Remark}

\crefrangeformat{equation}{#3(#1)#4--#5(#2)#6}

\crefname{thrm}{Theorem}{Theorems}
\crefname{lem}{Lemma}{Lemmas}
\crefname{cor}{Corollary}{Corollaries}
\crefname{prop}{Proposition}{Propositions}
\crefname{defn}{Definition}{Definitions}
\crefname{exm}{Example}{Examples}
\crefname{rem}{Remark}{Remarks}
\crefname{section}{Section}{Sections}
\crefname{equation}{\unskip}{\unskip}
\crefname{enumi}{\unskip}{\unskip}
\crefname{subsection}{Subsection}{Subsections}

\makeatletter
\newcommand{\mylabel}[2]{#2\def\@currentlabel{#2}\label{#1}}
\makeatother

\renewcommand{\iff}{\Leftrightarrow}

\newcommand{\impl}{\Rightarrow}

\newcommand{\id}{\mathrm{id}}

\newcommand{\I}{\mathcal I}

\newcommand{\cG}[1]{\mathcal G{(#1)}}
\newcommand{\D}{\mathcal D}
\newcommand{\U}[1]{\mathcal U{(#1)}}

\newcommand{\dom}[1]{\operatorname{\mathrm{dom}}{#1}}
\newcommand{\ran}[1]{\operatorname{\mathrm{ran}}{#1}}

\newcommand{\End}{\mathrm{End}}

\newcommand{\mend}{\operatorname{end}}
\newcommand{\minn}{\operatorname{in}}
\newcommand{\exo}{\operatorname{exo}}

\newcommand{\m}{{}^{-1}}

\newcommand{\0}{\theta}
\newcommand{\ve}{\varepsilon}

\newcommand{\af}{\alpha}
\newcommand{\bt}{\beta}

\newcommand{\f}{\varphi}

\newcommand{\dl}{\delta}

\newcommand{\M}{\mathcal M}

\newcommand{\tl}{\tilde}

\begin{document}

\title[$H^3(G,C(A))$ and extensions of $A$ by $G$]{The third partial cohomology group and existence of extensions of semilattices of groups by groups}
	\author{Mikhailo Dokuchaev}
	\address{Insituto de Matem\'atica e Estat\'istica, Universidade de S\~ao Paulo,  Rua do Mat\~ao, 1010, S\~ao Paulo, SP,  CEP: 05508--090, Brazil}
	\email{dokucha@gmail.com}
	\author{Mykola Khrypchenko}
	\address{Departamento de Matem\'atica, Universidade Federal de Santa Catarina, Campus Reitor Jo\~ao David Ferreira Lima, Florian\'opolis, SC,  CEP: 88040--900, Brazil \and Centro de Matem\'atica e Aplica\c{c}\~oes, Faculdade de Ci\^{e}ncias e Tecnologia, Universidade Nova de Lisboa, 2829--516 Caparica, Portugal}
	\email{nskhripchenko@gmail.com}
	\author{Mayumi Makuta}
	\address{Insituto de Matem\'atica e Estat\'istica, Universidade de S\~ao Paulo,  Rua do Mat\~ao, 1010, S\~ao Paulo, SP,  CEP: 05508--090, Brazil}
	\email{may.makuta@gmail.com} 
	\begin{abstract} We  introduce the concept of a partial  abstract kernel associated to a group $G$ and a semilattice of groups $A$ and relate the partial cohomology group $H^3(G,C(A))$ with the  obstructions to the existence of admissible extensions of $A$ by $G$ which realize the given   abstract kernel. We also show that if such extensions exist then they are classified by   $H^2(G,C(A)).$
\end{abstract}

	\maketitle
	
	\section*{Introduction}  Influenced by R. Exel's notion of a continuous twisted partial group action on a $C^*$-algebra   \cite{E0} and its  ring theoretic analogue in \cite{DES1},  a cohomology theory was introduced in \cite{DK}, which suits   unital twisted partial actions as well as the concept of their equivalence  given in \cite{DES2}. The cohomology from \cite{DK}   is strongly related to the cohomology of inverse semigroups. It found applications to partial projective group representations in  \cite{DK,DoSa}, to a generalization of the Chase-Harrison-Rosenberg se\-ven term exact sequence  for partial Galois extensions in \cite{DoPaPi,DoPaPiRo}, and to the study of ideals in reduced crossed products of $C^*$-algebras by global actions in \cite{KennedySchafhauser}. It  also influenced a Hopf theoretic treatment of partial cohomology in \cite{BaMoTe} and  a study of its affinity with extensions in \cite{DK2,DK3}. It became clear from \cite{DK2} that, in order to deal with extensions, a more general multiplier-valued partial group cohomology theory is needed. This was introduced in \cite{DK3} and appropriate interpretations for the first and second partial co\-ho\-mo\-lo\-gy groups in terms of extensions of semilattices of abelian groups by groups were obtained. More information on partial co\-ho\-mo\-lo\-gy and other developments around partial actions may be found in the survey ar\-ticle~\cite{D3}.\\

The goal of the present paper is to relate  the third partial cohomology group with obstructions to the existence of  extensions of  semilattices of non-necessarily abelian groups by  groups, which realize  partial  abstract kernels, and, if such extensions exist, classify them by the second partial  cohomology group.\\

   In Section~\ref{sec:Prelim} we give some notions and establish  preliminary facts on multipliers of semigroups, twisted partial group actions,  certain inverse monoids linked to the symmetric inverse monoid, relatively invertible endomorphisms of semilattices of groups, partial homomorphisms and premorphisms. In Section~\ref{sec:abskernextAbyG} we introduce the concept of a
(partial)  abstract kernel associated to a group $G$ and a semilattice of groups $A$ as a partial homomorphism $\psi :   G \to  \varsigma(A), $  where   $\varsigma(A)$ is the inverse monoid formed by certain equivalence classes of isomorphisms between ideals of $A.$  We  show that equivalent admissible extensions of $A$ by $G$ result in the same abstract kernel (see Proposition~\ref{prop:nunequivmesmokernelabs}), define the obstruction  to an admissible extension associated to an abstract kernel  $\psi $ (a partial $3$-cocycle), construct a partial action on the center $C(A)$ related to  $\psi $ and prove in Theorem~\ref{theo:main1} that   $\psi $  possesses an admissible extension if and only if the corresponding obstruction is the trivial element of the third partial cohomology group  $H^3(G,C(A)).$ In the final Subsection~\ref{sec-descr-ext} we show that if   $\psi $ possesses  an admissible extension then  the equivalence classes of admissible extensions realizing $\psi$ are in a one-to-one correspondence with the elements of $H^2(G, C(A))$ (see Theorem~\ref{theo:main2}).

	\section{Preliminaries}\label{sec:Prelim}
	
	\subsection{Multipliers}\label{sec:mult} 
	We shall use several definitions and technical results about multipliers, so all of them are gathered in this subsection.
	\begin{defn}\label{def:multiplicador}
		A \textit{multiplier} of a semigroup $S$ is a pair $m$ of maps $s \mapsto ms$ and $s \mapsto sm$ from $S$ to itself, such that for all $s,t \in S$:
		\begin{enumerate}[label=(\roman*)]
			\item $m(st) = (ms)t$;
			\item $(st)m = s(tm)$;
			\item $s(mt) = (sm)t$.
		\end{enumerate}
		The set of multipliers of $S$, denoted by $\M(S)$, forms a monoid under the usual composition of maps.
	\end{defn}
	
	\begin{rem}\label{m-and-n-associate-with-s} 
		\begin{enumerate}
			\item  Let $S^2=S$. Then for all $m,n\in\M(S)$ and $s\in S$
			\begin{align*}
			(ms)n=m(sn).
			\end{align*} 
			\item\label{ws=sw-for-s-in-C(S)} Let $C(S)^2=C(S)$. Then for all $m\in\M(S)$ and $s\in C(S)$
			\begin{align*}
				ms=sm.
			\end{align*}
		\end{enumerate}
	In particular, both items are true, if $S$ is inverse.
	\end{rem}
	\begin{proof}
		Item (i) is the semigroup version of~\cite[Proposition 2.5 (ii)]{DE}, while (ii) follows from~\cite[Remark 5.2]{DK2}. 
		Now, if $S$ is inverse, then $s=s(ss\m)\in S^2$ for any $s\in S$, so $S^2=S$. Moreover, $C(S)$ is an inverse subsemigroup of $S$, as for any $s\in C(S)$ and $t\in S$ one has $s\m t=s\m (t\m)\m=(t\m s)\m=(st\m)\m=ts\m$, whence $s\m\in C(S)$. Thus, $C(S)^2=C(S)$ also holds.
	\end{proof}

	\begin{defn}\label{def:innermult}
		Let $S$ be a semigroup. Define $\phi : S \to \M(S)$, $s \mapsto \phi_s$, to be the multiplier that acts on $t \in S$ as follows:
		\begin{align*}
			t \phi_s = ts,\ \phi_s t = st.
		\end{align*}
		The multiplier $\phi_s$ will be called the \textit{inner multiplier}\footnote{Petrich \cite{Petrich} calls this an \textit{inner bitranslation} induced by $s$.} associated with the element $s$ of $S$.
	\end{defn}
	
	\begin{defn}{\cite[p. 3282]{DES1}}\label{def:malpha}
		Let $\af : S \to T$ be an isomorphism of semigroups. For any $m\in\M(S)$ denote by $m^\af$ the multiplier of $T$ acting on $t\in T$ in the following way:
		\begin{align}\label{tm^af=af(af-inv(t)m)}
			t m^\af = \af(\af\m(t) m),\ m^\af t = \af(m \af\m(t)).
		\end{align}
	\end{defn}

	\begin{rem}\label{m->m^af-iso}
		It is easily seen that $m\mapsto m^\af$ is an isomorphism of monoids $\M(S)\to\M(T)$.
	\end{rem}

	\begin{defn}\label{defn-conj-by-multiplier}
		Let $S^2=S$ and $m\in\U{\M(S)}$. Denote by $\mu(m)$ the conjugation $s\mapsto msm\m$ on $S$. By \cref{m-and-n-associate-with-s} it is well-defined. Moreover, it is clearly an automorphism of $S$.
	\end{defn}

	\begin{lem}\label{lem:malpha}
		Let $\af:S\to T$ be an isomorphism of semigroups.
		\begin{enumerate}[label=(\alph*)]
			\item If $\bt:T\to U$ is another semigroup isomorphism, then
			\begin{align}\label{eq:mbetaalpha}
			(m^\af)^\beta = m^{\beta \circ \af}
			\end{align}
			for any $m\in\M(S)$.
			\item If $S^2=S$, then
			\begin{align}\label{eq:malphaconj}
			\mu(m^\af)=\af \circ \mu(m) \circ \af\m
			\end{align}
			for any $m\in\U{\M(S)}$.
		\end{enumerate} 
	\end{lem}
	\begin{proof}
		(a) Given $s \in S$, we have
			\begin{align*}
				s (m^\af)^\beta & =  \beta(\beta\m(s)m^\af)  =  \beta(\af(\af\m(\beta\m(s))m)) \\
				& =  (\beta \circ \af)((\beta \circ \af)\m(s)m)  =  s m^{\beta \circ \af}.
			\end{align*}
			Similarly, $(m^\af)^\beta s = m^{\beta \circ \af} s$.
			
			(b) For arbitrary $t \in T$ we have
			\begin{align*}
			\mu(m^\af)(t) & =  (m^\af t) (m^\af)\m =  \af(\af\m(\af(m\af\m(t)))m\m) \\
			  & =  \af( m \af\m(t) m\m ) =  (\af \circ \mu(m) \circ \af\m)(t).
			\end{align*}
	\end{proof}
	
	\begin{lem}\label{m-in-C(M(S))-iff-ms=sm}
		Let $S^2=S$ and $m \in \M(S)$. Then
		\begin{align*}
			m \in C(\M(S)) \iff \forall s \in S:\ ms = sm.
		\end{align*}
	\end{lem}
	\begin{proof}
		The \textit{``if''} part. Let $m \in C(\M(S))$. Since $S^2=S$, any element $s\in S$ is of the form $s = tu$ for some $t,u \in S$. Clearly,  $s = \phi_tu = t\phi_u$, where $\phi_t$ and $\phi_u$ are the inner multipliers from \Cref{def:innermult}. Now,
		\begin{align*}
		ms  =  mtu  =  m \phi_t u  =  \phi_t m u  =  t m \phi_u  =  t \phi_u m  =  tu m =  sm.
		\end{align*}
		
		The \textit{``only if''} part. Suppose that $m$ ``commutes'' with any element of $S$. Given $n \in \M(S)$ and $s\in S$, we calculate using \cref{m-and-n-associate-with-s}(i) 
		\begin{align*}
			(mn)s  =  m(ns)  =  (ns)m =  n(sm) = n(ms) = (nm)s.
		\end{align*}
		Similarly $s(mn)=s(nm)$. Thus, $mn=nm$.
	\end{proof}
	
	\begin{lem}\label{C(M(S))-sst-M(C(S))}
		Let $S^2=S$. Then
		 \begin{align*}
		 	m \in C(\M(S)) \impl m \in \M(C(S)).
		 \end{align*}
	\end{lem}
	\begin{proof}
		Given $m \in C(\M(S))$ and $s \in C(S)$, we know by \cref{m-in-C(M(S))-iff-ms=sm} that $ms=sm$. So, it suffices to show that $ms\in C(S)$. For any $t \in S$ using \cref{m-in-C(M(S))-iff-ms=sm} we have:
		\begin{align*}
			(ms)t  =  m(st) =  m(ts) = (ts)m = t(sm) = t(ms).
		\end{align*}	
	\end{proof}
	
	\begin{cor}
		Under the conditions of \cref{C(M(S))-sst-M(C(S))} if $m \in \U{C(\M(S))}$, then $m \in \U{\M(C(S))}$.
	\end{cor}

	\begin{lem}\label{lem:existbeta}
		Let $S^2=S$ and $m,n \in \U{\M(S)}$. Then $\mu(m) = \mu(n)$ if and only if there exists $w \in \U{C(\M(S))}$ such that 
		\begin{align}\label{n-is-mn}
			m=wn.
		\end{align}
	\end{lem}
	\begin{proof}
		The \textit{``if''} part is obvious, so we shall only prove the \textit{``only if''} part. By definition $\mu(m) = \mu(n)$ means that 
		\begin{align*}
			msm\m = nsn\m
		\end{align*}
		for all $s \in S$. So, defining $w : = mn\m\in \U{\M(S)}$, we clearly have \cref{n-is-mn} and moreover
		\begin{align*}
			ws  = sw
		\end{align*}
		for arbitrary $s \in S$. It follows by \Cref{m-in-C(M(S))-iff-ms=sm} that $w \in C(\M(S))$.
	\end{proof}
	
	\begin{lem}\label{lem:trocavw}
		Let $S$ be such that $C(S)^2=C(S)$. Given $m \in \M(C(S))$ and $n \in \M(S)$, we have
		\begin{align*}
			mn = nm \mbox{ on } C(S).
		\end{align*}
	\end{lem}
	\begin{proof}
		Let $s \in C(S)$. Since $m \in \M(C(S))$, it follows that $ms\in C(S)$.  Then using \cref{m-and-n-associate-with-s}(ii) we have
		\begin{align*}
			(nm)s  =  n(ms)  =  (ms)n = m(sn) = m(ns) = (mn)s.
		\end{align*}
		Equality $s(mn)=s(nm)$ is proved similarly.
	\end{proof}
	
	\begin{lem}\label{lem:mgthetatroca}
		Let $S^2=S$ and $\psi,\f : S \to T$ be two isomorphisms of semigroups. If $\psi = \mu(m) \circ \f$ for some $m \in \M(T)$, then  
		\begin{equation}\label{eq:mgthetatroca}
		mn^{\f} = n^{\psi}m
		\end{equation} 
		for all $n \in \M(S)$.
	\end{lem}
	\begin{proof}
		Given $s \in S$, using \cref{m-and-n-associate-with-s}(i) we have
		\begin{align*}
			(n^{\psi}m)s  & = n^{\psi}(ms) = n^{\mu(m) \circ \f}(ms) =  (\mu(m) \circ \f)(n(\mu(m) \circ \f)\m(ms))\\
			&  =  m\f(n(\mu(m) \circ \f)\m(ms))m\m  =  m\f(n\f\m(m\m msm))m\m \\
			& =  mn^{\f}(sm)m\m  =  m(n^{\f}s)mm\m  =  (mn^{\f})s.
		\end{align*}
		Similarly, $s(n^{\psi}m)=s(mn^{\f})$.
	\end{proof}
	
	\begin{prop}\label{prop:defv'}
		Let $A$ be a semilattice of groups. Then there is a one-to-one correspondence between the elements of $\M(C(A))$ and the elements of $C(\M(A))$. 
	\end{prop} 
	\begin{proof}
		By \cref{C(M(S))-sst-M(C(S))} we know that each $m\in C(\M(A))$ is a multiplier of $C(A)$. It remains to show that each $m \in \M(C(A))$ uniquely extends to $m'\in C(\M(A))$. Given such $m$, we shall first prove the following equality
		\begin{align}\label{m(aa-inv)a=a(a-inv.a)m}
		m(aa\m)a = a(a\m a)m
		\end{align}
		and then use it to define an extension of $m$ to a central multiplier on the whole $A$.
		
		Since $A$ is a semilattice of groups, we have $aa\m = a\m a \in C(A)$, so both sides of \cref{m(aa-inv)a=a(a-inv.a)m} are well-defined, and moreover $m(aa\m),(a\m a)m\in C(A)$. Hence, using \cref{m-and-n-associate-with-s} we get
		\begin{align*}
			m(aa\m)a = am(aa\m) = a(aa\m)m = a(a\m a)m,
		\end{align*} 
		proving \cref{m(aa-inv)a=a(a-inv.a)m}. We now define 
		\begin{align}\label{m'a=m(aa-inv)a}
			m'a:=m(aa\m)a.
		\end{align}
		Then
		\begin{align*}
			m'(ab) & =  m((ab)(ab)\m)ab = m(abb\m a\m)ab = m(aa\m bb\m)ab\\
			 &= m(aa\m) bb\m a b = m(aa\m) a bb\m b = m(aa\m) a b = (m'a)b,
		\end{align*}
		and by symmetry if we put
		\begin{align}\label{am'=a(a-inv-a)m}
			am':= a(a\m a)m,
		\end{align}
		then $(ab)m' = a(bm')$. Furthermore,
		\begin{align*}
			(am')b  =  a(a\m a)mb = a(a\m a)mbb\m b = aa\m am(bb\m)b = a m(bb\m)b = a(m'b).
		\end{align*}
	Thus, $m'\in\M(A)$. That $m'\in C(\M(A))$ follows by \cref{m(aa-inv)a=a(a-inv.a)m,m-in-C(M(S))-iff-ms=sm}. Clearly, the extension of $m$ is unique, as each $n\in\M(A)$ satisfies $na=n(aa\m)a$ and $an=a(a\m a)n$.
	\end{proof}
	
	\subsection{Twisted partial actions}
	\begin{defn}
		A \textit{twisted partial action}~\cite{DES1,DK2} of a group $G$ on a semigroup $A$ is a pair $\Theta=(\0,w)$, where $\0=\{\0_g:\D_{g\m}\to\D_g\mid g\in G\}$ is a collection  of isomorphisms between two-sided ideals of $A$ and $w=\{w_{g,h}\in\M(\D_g\D_{gh})\mid g,h\in G\}$ satisfying the following properties for all $g, h, k \in G$:
		\begin{enumerate}[label=(TPA\arabic*),leftmargin=2cm]
			\item\label{TPA1} $\D^2_g = \D_g$, $\D_g\D_h = \D_h\D_g$;
			\item\label{TPA2} $\D_1 = A$, $\0_1 = \id_A$;
			\item\label{TPA3} $\0_g(\D_{g\m}\D_h) = \D_g\D_{gh}$;
			\item\label{TPA4} $(\0_g \circ \0_h) (a) = w_{g,h} \0_{gh}(a) w\m_{g,h}$ for all $a \in \D_{h\m}\D_{h\m g\m}$;
			\item\label{TPA5} $w_{1,g} = w_{g,1} = \id_{\D_g}$;
			\item\label{TPA6} $\0_g(aw_{h,k})w_{g,hk} = \0_g(a)w_{g,h} w_{gh,k}$ for all $a\in\D_{g\m}\D_h\D_{hk}$.
		\end{enumerate}
	\end{defn}
	
	\begin{defn}
		A \textit{partial action}~\cite{DN,DK} of a group $G$ on a semigroup $A$ is a collection $\0=\{\0_g:\D_{g\m}\to\D_g\mid g\in G\}$ of isomorphisms between two-sided ideals of $A$ satisfying the following properties for all $g, h \in G$:
		\begin{enumerate}[label=(PA\arabic*),leftmargin=2cm]
			\item\label{PA1} $\D_1 = A$, $\0_1 = \id_A$;
			\item\label{PA2} $\0_g(\D_{g\m} \cap \D_h) = \D_g \cap \D_{gh}$;
			\item\label{PA3} $(\0_g \circ \0_h) (a) = \0_{gh}(a)$, $\forall a \in \D_{h\m} \cap \D_{h\m g\m}$.
		\end{enumerate}
	\end{defn}
	
	\begin{defn}\label{defn-equiv-tw-pact}
		Two twisted partial actions $(\0,w)$ and $(\0',w')$ of $G$ on $A$ are called {\it equivalent}, if
		\begin{enumerate}[label=(ETPA\arabic*),leftmargin=2cm]
			\item\label{ETPA1} $\D'_g=\D_g$ for all $g\in G$;
		\end{enumerate}
		and there exists $\ve=\{\ve_g\in\U{\M(\D_g)}\mid g\in G\}$, such that
		\begin{enumerate}[label=(ETPA\arabic*),leftmargin=2cm]\setcounter{enumi}{1}
			\item\label{ETPA2} $\0'_g(a)=\ve_g\0_g(a)\ve\m_g$ for all $g\in G$ and $a\in\D_{g\m}$;
			\item\label{ETPA3} $\0'_g(a)w'_{g,h}\ve_{gh}=\ve_g\0_g(a\ve_h)w_{g,h}$ for all $g,h\in G$ and $a\in\D_{g\m}\D_h$.
		\end{enumerate}
	\end{defn}

	\subsection{The quotient \texorpdfstring{$\varsigma(A)$}{s(A)}}\label{sec:varsigmaA} Throughout this subsection $A$ will be a semilattice of (not necessarily abelian) groups. Given a set $X$, by $\I(X)$ we denote the symmetric inverse monoid~\cite[p. 6]{Lawson} of $X$.
	
	\begin{rem}\label{I^2=I-in-inverse-S}
		If $S$ is an inverse semigroup, then each ideal $I$ of $S$ is idempotent, since each $s \in I$ equals $s\cdot s\m s$ with $s\m s \in I$. It follows that $I \cap J = IJ$ for any two non-empty ideals of $S$, since $I \cap J = (I \cap J)^2 \subseteq IJ$, and the converse inclusion holds in an arbitrary semigroup. In particular, any two non-empty ideals of $S$ commute.
	\end{rem}
	
	\begin{lem}\label{I-lhd-J-lhd-S=>I-lhd-S}
		Let $S$ be a semigroup and $I \lhd J \lhd S$. If $I^2=I$, then $I \lhd S$.
	\end{lem}
	\begin{proof}
		Indeed, $IS=I^2S=I\cdot IS\subseteq I\cdot JS\subseteq IJ\subseteq I$, and similarly $SI\subseteq I$.
	\end{proof}
	
	\begin{defn}\label{def:SigmaA}
		Given a semigroup $S$, define $\Sigma(S)$ to be the subset of $\I(S)$ that consists of isomorphisms between ideals of $S$.
	\end{defn}

	\begin{prop}
		Let $S$ be an inverse semigroup. Then $\Sigma(S)$ is an inverse submonoid of $\I(S)$.
	\end{prop}
	\begin{proof}
		Given $\f,\psi\in\Sigma(S)$, we have: 
		\begin{align*}
			\f \circ \psi: \psi\m(\ran\psi \cap \dom\f) \to \f(\ran\psi \cap \dom\f).
		\end{align*}
		Observe that $\ran\psi \cap \dom\f\lhd\ran\psi$, so $\dom (\f \circ \psi) \lhd\dom\psi$, as $\psi$ is an isomorphism between $\dom\psi$ and $\ran\psi$. Similarly $\ran (\f \circ \psi) \lhd \ran\f$. By \Cref{I-lhd-J-lhd-S=>I-lhd-S} both $\dom (\f \circ \psi)$ and $\ran (\f \circ \psi)$ are ideals of $S$, and hence $\f \circ \psi\in\Sigma(S)$.
		
		It is obvious that $\Sigma(S)$ is closed with respect to the inverses and contains the identity element of $\I(S)$.
	\end{proof}

	\begin{defn}\label{def:sim}
		Let $S$ be an inverse semigroup. Define the following binary relation $\sim$ on $\Sigma(S)$: 
		\begin{align*}
			\psi\sim\f \iff 
			\begin{cases}
				\dom \f = \dom \psi\mbox{ and }\ran \f = \ran \psi,\\
				\exists m \in \U{\M(\ran\f)}\,\forall s\in\dom\f:\ \psi(s) = \ m\f(s)m\m.
			\end{cases}
		\end{align*}
	\end{defn}

	By \cref{m-and-n-associate-with-s,I^2=I-in-inverse-S} the expression involving the multiplier $m$ in \cref{def:sim} makes sense. It is easy to see that $\sim$ is an equivalence: reflexivity and symmetry are obvious, and the transitivity follows from the fact that $\M(\ran\f)$ is a monoid.	We are going to prove that $\sim$ is in fact a congruence. To this end, we shall need a technical lemma.
	
	\begin{lem}\label{lem:relpreservideais}
		Let $S$ be an inverse semigroup. Given $\psi\sim\f$ and $I \lhd \dom\f$, $J \lhd \ran\f$, one has
		\begin{align*}
			\f(I) = \psi(I) \mbox{ and } \f\m(J) = \psi\m(J).
		\end{align*}
	\end{lem}
	\begin{proof}
		Let $m \in \U{\M(\ran\f)}$ be such that $\psi(s) = m\f(s)m\m$ for all $s \in\dom\f$.
		
		We first prove the assertion for the ideal $I \lhd \dom\f$. Since $\f(I)$ is idempotent thanks to \cref{I^2=I-in-inverse-S}, it follows by~\cite[Remark 5.3]{DK2} that $m \f(I)m\m = \f(I)$. But $\psi(I) = m\f(I)m\m$, since $\f\sim\psi$, whence $\f(I) = \psi(I)$.

		Now, any $J \lhd\ran\f$ equals $\f(\f\m(J))$, the latter being $\psi(\f\m(J))$ by the first equality of the lemma applied to $I=\f\m(J)\lhd \dom\f$. Hence, $\psi\m(J) = \psi\m(\psi(\f\m(J)))=\f\m(J)$.
	\end{proof}

	\begin{prop}\label{prop:simcong}
		Let $S$ be an inverse semigroup. Then the equivalence $\sim$ from \Cref{def:sim} is a congruence on $\Sigma(S)$.
	\end{prop}
	\begin{proof}
		We shall prove that $\sim$ is both left and right compatible. Let $\f\sim\psi$ and $\af \in \Sigma(S)$. We need to show that $\af \circ \f\sim\af \circ \psi$ and $\f \circ \af\sim\psi \circ \af$.
		
		Recall that
		\begin{align*}
			\af \circ \f &: \f\m(\ran\f\cap\dom\af) \to \af(\ran\f\cap\dom\af),\\
			\af \circ \psi &: \psi\m(\ran\psi\cap\dom\af) \to \af(\ran\psi\cap\dom\af).
		\end{align*}
		The equality $\ran (\af \circ \f) = \ran (\af \circ \psi)$ is clear, as $\ran\f=\ran\psi$. Furthermore, since $\ran\f\cap\dom\af=\ran\psi\cap\dom\af$ is an ideal of $\ran\f=\ran\psi$, using \Cref{lem:relpreservideais} we have
		\begin{align*}
			\dom (\af \circ \f)&=\f\m(\ran\f\cap\dom\af) = \psi\m(\ran\f\cap\dom\af)\\
			&= \psi\m(\ran\psi\cap\dom\af)=\dom (\af \circ \psi).
		\end{align*}
		Now let $m\in\U{\M(\ran\f)}$ such that $\psi(s) = m\f(s)m\m$ for all $s\in\dom\f$. In view of~\cite[Remark 5.3]{DK2} the multiplier $m$ of $\ran\f$ restricts to a multiplier of $\ran\f\cap\dom\af\lhd\ran\f$, which we denote using the same letter. The isomorphism $\af$ restricted to $\ran\f\cap\dom\af$ induces the invertible multiplier $m^\af$ of $\af(\ran\f\cap\dom\af)$ as in \Cref{def:malpha}.
		Then
		\begin{align*}
			m^\af (\af \circ \f)(x) (m^\af)\m & =  \af(m(\af\m\circ\af \circ \f)(x))(m^\af)\m =  \af(m \f(x)) (m^\af)\m \\
			& =  \af((\af\m\circ \af)(m \f(x))m\m) =  \af( m \f(x) m\m ) =  (\af \circ \psi)(x),
		\end{align*}
		proving that $\af \circ \f\sim\af \circ \psi$.
		
		The right compatibility of $\sim$ is even easier to prove. We have
		\begin{align*}
			\psi \circ \af &: \af\m(\ran\af\cap\dom\psi) \to \psi(\ran\af\cap\dom\psi),\\
			\f \circ \af &: \af\m(\ran\af\cap\dom\f) \to \f(\ran\af\cap\dom\f).
		\end{align*}
		It is evident that $\dom (\psi \circ \af) = \dom (\f \circ \af)$. Since $\ran\af\cap\dom\f\lhd\dom\f$, we get $\f(\ran\af\cap\dom\f) = \psi(\ran\af\cap\dom\f)$ by \Cref{lem:relpreservideais}, yielding $\ran (\psi \circ \af) = \ran (\f \circ \af)$. Finally, if $s\in \dom(\f \circ \af)$, then $\af(s)\in\dom\f$, so $\psi(\af(s)) = m \f(\af(s)) m\m$, and thus $\f \circ \af\sim\psi \circ \af$.
	\end{proof}
	
	\begin{defn}\label{defn:sigmaAmon}
		Let $S$ be an inverse semigroup. Then the quotient $\Sigma(S)/{\sim}$ is an inverse monoid, which will be denoted by $\varsigma(S)$.
	\end{defn}
	
	\begin{rem}
		The projection homomorphism $\Sigma(S)\to\varsigma(S)$ is idempotent-separating.
	\end{rem}
	\begin{proof}
		Indeed, any idempotent of $\Sigma(S)$ is of the form $\id_I$, where $I$ is an ideal of $S$. Suppose that $[\id_I]=[\id_J]$ for some ideals $I,J$. Then $\id_I\sim\id_J$, which implies that $I=\dom(\id_I)=\dom(\id_J)=J$. Thus, $\id_I=\id_J$.
	\end{proof}

	Recall the following definition from~\cite{Lausch}.
	\begin{defn}
		Let $A$ be a semilattice of groups. An endomorphism $\f : A \to A$ is called \textit{relatively invertible} if there exist $\bar{\f}\in \End A$ and $e_\f \in E(A)$ satisfying:
		\begin{enumerate}[label=(\roman*)]
			\item $\bar{\f} \circ \f (a) = e_\f a$ and $\f \circ \bar{\f}(a) = \f(e_\f) a$, for any $a \in A$;
			\item $e_\f$ is the identity of $\bar{\f}(A)$ and $\f(e_\f)$ is the identity of $\f(A)$.
		\end{enumerate}
		The set of relatively invertible endomorphisms of $A$ is denoted by $\mend A$.
	\end{defn}
	
	\begin{prop}{\cite[Proposition 3.4]{DK2}}\label{prop:3.4}
		The set $\mend A$ forms an inverse subsemigroup of $\Sigma(A)$ isomorphic to $\I_{ui} (A)$, the semigroup of isomorphisms between unital ideals of $A$.
	\end{prop}
	
	\begin{defn}
		For each fixed $a \in A$, the map $\mu_a(b) = aba\m$ is an endomorphism called \textit{inner}. It is relatively invertible with $\overline{\mu_a} = \mu_{a\m}$ and $e_{\mu_a} = a\m a= aa\m$. The set of all relatively invertible inner endomorphisms forms an inverse subsemigroup of $\mend A$ denoted by $\minn A$.  
	\end{defn}
	
	\begin{defn}
		The semigroup $\minn A$ is a semilattice of groups which determines a kernel normal system of $\mend A$. The quotient $\mend A/{\minn A}$  will be denoted by $\exo A$ and its elements will be called \textit{exomorphisms} of $A$. Observe that the projection epimorphism $\mend A\to\exo A$ is idempotent-separating.
	\end{defn}
	
	The following remark shows that $\sim$ is a generalization of the congruence $\minn A$ to the case of isomorphisms between non-unital ideals of an inverse semigroup.
	\begin{rem}
		Let $A$ be a semilattice of groups and $\f,\psi\in\Sigma(A)$ isomorphisms between \textit{unital} ideals of $A$. Then, identifying $\f$ and $\psi$ with elements of $\mend A$ as in~\cite[Proposition 3.4]{DK2}, we have $\f\sim\psi$ if and only if $(\f,\psi)\in\minn A$. Consequently, $\exo(A)$ can be seen as an inverse subsemigroup of $\varsigma(A)$. 
	\end{rem}

	
	\begin{rem}\label{equiv-Thetas}
		Let $G$ be a group, $S$ an inverse semigroup and $\Theta=(\0,w)$, $\Theta'=(\0',w')$ two twisted partial actions of $G$ on $S$. If $\Theta$ is equivalent to $\Theta'$ then $\0_g\sim\0'_g$ for all $g\in G$.
	\end{rem}
	
	\subsection{Partial homomorphisms}\label{sec:phomprem}

	\begin{defn}\label{def:phom}
		A map $f : G \to S$ from a group $G$ to a semigroup $S$ is called a \textit{partial homomorphism} if it satisfies, for all $g,h \in G$:
		\begin{enumerate}[label=(PH\arabic*)]
			\item\label{PH1} $f(g\m) f(g) f(h) = f(g\m) f(gh)$;
			\item\label{PH2} $f(g) f(h) f(h\m) = f(gh) f(h\m)$;
			\item\label{PH3} $f(g) f(1) = f(g)$.
		\end{enumerate}
		If $S$ is a monoid, then the partial homomorphism is said to be \textit{unital} if, instead of \cref{PH3}, one has $f(1) = 1$.
	\end{defn}
	
	\begin{rem}
		\cref{def:phom} is a slight modification of what was defined in~\cite{DN}, where $S$ was supposed to be a monoid. 
	\end{rem}
	
	The following is well-known and easy to prove.
	\begin{cor}\label{propert-of-part-hom}
		Let $f:G\to S$ be a partial homomorphism. Then for all $g,h \in G$
		\begin{enumerate}
			\item $f(g)f(g\m)f(g)=f(g)$;
			\item $f(1)f(g)=f(g)$;
			\item $f(g)f(h)=f(g)f(g\m)f(gh)=f(gh)f(h\m)f(h)$.\label{f(g)f(h)=f(g)f(g-inv)f(gh)}
		\end{enumerate}
	In particular, if $S$ is inverse, then $f(g\m)=f(g)\m$.
	\end{cor}
	Within this subsection we shall fix a group $G$, a semilattice of groups $A$ and a partial homomorphism $\psi : G \to \varsigma(A)$. For any $g\in G$ we denote by $\D_g$ the common range of all the isomorphisms of the class $\psi(g)$. We shall also choose a set of representatives $\0_g\in\psi(g)$.
	
	\begin{lem}\label{domain-of-0_g}
		The common domain of all the isomorphisms from $\psi(g)$ is $\D_{g\m}$. In particular, $\0_g$ is an isomorphism $\D_{g\m}\to\D_g$.
	\end{lem}
	\begin{proof}
		Consider the composition $\0_g\0_{g\m}\in\psi(g)\psi(g\m)$. Since $\psi(g)\psi(g\m)=\psi(g)\psi(g)\m$ is an idempotent, it is the class of some idempotent $\id_I$ from $\Sigma(A)$. Hence, $\dom(\0_g\0_{g\m})=\ran(\0_g\0_{g\m})=I$. We shall first prove that $I=\D_g$. Using the fact that $\psi(g)\psi(g\m)\psi(g)=\psi(g)$, we have $\ran(\0_g\0_{g\m}\0_g)=\ran(\0_g)=\D_g$. It follows that $\D_g\subseteq\ran(\0_g\0_{g\m})=I$. On the other hand, $I=\ran(\0_g\0_{g\m})\subseteq\ran(\0_g)=\D_g$, whence $I=\D_g$, as desired.
		
		Clearly, $\dom(\0_g)=\dom(\0_g\0_{g\m}\0_g)\subseteq\dom(\0_{g\m}\0_g)=\D_{g\m}$. On the other hand, $\D_{g\m}=\dom(\0_{g\m}\0_g)\subseteq\dom(\0_g)$, proving the lemma.
	\end{proof}
	
	\begin{lem}\label{lem:tpa3}
		For any $g\in G$ one has
		\begin{align*}
		\0_g(\D_{g\m}\D_h) = \D_g\D_{gh}.
		\end{align*}
	\end{lem}
	\begin{proof}
		Indeed, using \cref{domain-of-0_g} we see that $\0_g(\D_{g\m}\D_h)=\0_g(\D_{g\m}\cap \D_h)=\ran (\0_g\0_h)$. But $\0_g\0_h\in\psi(g)\psi(h)=\psi(g) \psi(g\m) \psi(gh)$ and $\id_{\D_g}\in\psi(g)\psi(g\m)$. So, $\ran(\0_g\0_h)=\ran(\id_{\D_g}\0_{gh})=\D_g\cap\D_{gh}=\D_g\D_{gh}$.
	\end{proof}
	
	\begin{cor}\label{cor:gentpa3}
		We also have a more general equality: \[\0_g(\D_{g\m} \D_{h_1} \ldots \D_{h_n}) = \D_g\D_{gh_1} \ldots \D_{gh_n}.\]
	\end{cor}
	\begin{proof}
		For by \cref{I^2=I-in-inverse-S} all $\D_g$ are idempotent and commute.
	\end{proof}

	\begin{cor}\label{cor:imidemps}
		Let $\psi(G)\subseteq\exo A$ and $\D_g=1_gA$ for some $1_g\in E(C(A))$. Then
		\begin{align*}
		\0_g(1_{g\m}1_h) = 1_g1_{gh}.
		\end{align*} 
	\end{cor}
	\begin{proof}
		Indeed, $1_{g\m}1_h$ is the identity element of $\D_{g\m}\D_h$ and $1_g1_{gh}$ is the identity element of $\D_g\D_{gh}$.
	\end{proof}

	\begin{lem}\label{lem:distcentro}
		Let $I$ and $J$ be ideals of a semilattice of groups $A$. Then
		\begin{align*}
		C(IJ) = C(I)C(J).
		\end{align*}
	\end{lem}
	\begin{proof}
		The inclusion $C(I)C(J)\subseteq C(IJ)$ holds for an arbitrary inverse semigroup $A$. Indeed, given $x \in C(I)$, $y \in C(J)$, and $a \in I \cap J = IJ$, we have: 
		\begin{align*}
		(xy)a  =  x(ya) =  x(ay) =  (xa)y =  (ax)y =  a(xy).
		\end{align*}
		
		For the converse inclusion $C(IJ) \subseteq C(I)C(J)$, we take $x \in I$ and $y \in J$ such that $xy \in C(IJ)$. Then, for any $a \in I$, using the fact that $E(A)\subseteq C(A)$, we have:
		\begin{align*}
		(xy)a  =  xy(y\m y)a =  xy(ay\m y) = (ay\m y)xy =  a xy (y\m y)= a(xy).
		\end{align*}
		Hence, $xy\in C(I)$. Similarly, $xy\in C(J)$. Thus, $xy \in C(I) \cap C(J) = C(I)C(J)$, proving $C(IJ) \subseteq C(I)C(J)$.
	\end{proof}
	
	\begin{cor}\label{cor:centroidemp}
		The center of an idempotent ideal $I$ of $A$ is also idempotent.
	\end{cor}
	\begin{proof}
		Indeed, if $I \lhd A$ and $I=I^2$, then $C(I) = C(I^2) = C(I)^2$.
	\end{proof}
	
	\begin{cor}\label{cor:ideaiscentro}
		For any $g\in G$ we have
		\begin{align*}
		\0_g(C(\D_{g\m})C(\D_h)) = C(\D_g)C(\D_{gh}).
		\end{align*} 
	\end{cor}
	\begin{proof}
		Since each $\0_g$ is an isomorphism, it preserves the centers, so using \cref{lem:distcentro,lem:tpa3}, we conclude that
		\begin{align*}
		\0_g(C(\D_{g\m})C(\D_h))=\0_g(C(\D_{g\m}\D_h))=C(\D_g\D_{gh})=C(\D_g)C(\D_{gh}).
		\end{align*} 
	\end{proof}

	\section{Abstract kernel of an extension of \texorpdfstring{$A$}{A} by \texorpdfstring{$G$}{G}}\label{sec:abskernextAbyG}
	
	\subsection{Abstract kernel of an extension of \texorpdfstring{$A$}{A} by \texorpdfstring{$S$}{S}}\label{sec-abstr-kern-of-ext-of-A-by-S}
	
	Recall Lausch's construction of the abstract kernel of an extension. Let $A\xrightarrow{i} U \xrightarrow{j} S$ be an extension of a semilattice of (not necessarily abelian) groups $A$ by an inverse semigroup $S$. Then (see \cite[p. 291]{Lausch}) there is an idempotent-separating homomorphism $\nu:U \to \mend A$, $u \mapsto \nu_u$, defined by 
	\begin{align*}
	\nu_u(a)=i\m(ui(a)u\m).
	\end{align*}
	Recall from \cite[Theorem 7.55]{Clifford-Preston-2} that $j(u)=j(v)$ for $u,v\in U$ if and only if there exists $a\in A$ such that $u=i(a)v$ and $i(aa\m)=vv\m$. Then 
	$$
	\nu_u=\nu_{i(a)}\circ\nu_v,
	$$
	and $\nu_{i(a)}\in\minn A$ with $\nu_{i(a)}\circ(\nu_{i(a)})\m=\nu_{i(aa\m)}=\nu_{vv\m}=\nu_v\circ(\nu_v)\m$. Therefore, $\nu_u$ and $\nu_v$ belong to the same class of the idempotent-separating congruence on $\mend A$ whose kernel coincides with $\minn A$ (\cite[Proposition 8.2 (ii)]{Lausch}). It follows that the map $\psi:S\to\exo A$, which sends $s\in S$ to the $\minn A$-class $[\nu_u]$ of $\nu_u$, where $j(u)=s$, is well-defined, in the sense that the definition does not depend on the choice of $u$. Moreover, for $s=j(u)$ and $t=j(v)$, one has $st=j(uv)$, and hence
	$$
	\psi(st)=[\nu_{uv}]=[\nu_u\circ\nu_v]=[\nu_u][\nu_v]=\psi(s)\psi(t),
	$$
	so $\psi$ is a homomorphism. If $\psi(e)=\psi(e')$ for some $e,e'\in E(S)$, then $[\nu_f]=[\nu_{f'}]$ for some (uniquely determined) pair $f,f'\in E(U)$ with $j(f)=e$ and $j(f')=e'$. Since $\nu_f,\nu_{f'}\in E(\mend A)$ and the natural epimorphism $\mend A\to\exo A$ is idempotent-separating by \cite[Proposition 8.2 (ii)]{Lausch}, we have $\nu_f=\nu_{f'}$, which implies $f=f'$, as $\nu$ is idempotent-separating. Thus, $\psi$ is also idempotent-separating as a homomorphism $S\to\exo A$.
	
	\subsection{From the abstract kernel of an extension of \texorpdfstring{$A$}{A} by \texorpdfstring{$S$}{S} to an abstract kernel of an extension of \texorpdfstring{$A$}{A} by \texorpdfstring{$G$}{G}}\label{sec-from-abstr-kern-A-S-to-abstr-kern-A-G}
	
	Now suppose that the extension $A\xrightarrow{i}U\xrightarrow{j}S$ above admits an order-preserving transversal $\rho:S\to U$, i.e. $j\circ\rho=\id_S$. Then by formulas (7), (9) and (11)  from \cite{DK2} the maps $\af:E(S)\to E(A)$, $\lambda:S\to\mend A$ and $f:S^2\to A$, where
	\begin{align*}
	\af(s)&=i\m(\rho(s)),\\
	\lambda_s&=\nu_{\rho(s)},\\
	f(s,t)&=i\m(\rho(s)\rho(t)\rho(st)\m),
	\end{align*}
	define a twisted $S$-module structure on $A$. Observe from the definition of $\psi$ that
	\begin{align}\label{psi(s)=[lambda_s]}
	\psi(s)=[\nu_{\rho(s)}]=[\lambda_s].
	\end{align}
	In particular, a choice of another order-preserving transversal $\rho'$ leads to the same class $[\lambda_s]$ in $\exo A$ for all $s\in S$ (see also \cite[Proposition 3.10]{DK2}).
	
	Furthermore, we know by \cite[Proposition 3.23]{DK2} that the twisted $S$-module $\Lambda=(\af,\lambda,f)$ is Sieben, so by \cite[Proposition 6.11]{DK2} it induces a twisted partial action $\Theta=(\0,w)$ of $\cG S$ on $A$, where
	\begin{align*}
	&\0_g(a)=\lambda_s(a),\mbox{ with }s\in g,\ \af(s\m s)=aa\m,\\
	&\begin{cases}
	w_{g,h}a&=f(s,s\m t)a\\
	aw_{g,h}&=af(s,s\m t)
	\end{cases}
	\mbox{ with }s\in g,\ t\in gh,\ \af(ss\m)=\af(tt\m)=aa\m.
	\end{align*}
	Moreover, thanks to~\cite[Proposition 6.13]{DK2} any $\Lambda'$ equivalent to $\Lambda$ induces $\Theta'$ equivalent to $\Theta$. Thus, in view of \cref{equiv-Thetas}, we have a well-defined map that sends $g \in \cG S$ to $[\0_g]\in\varsigma(A)$.
	
	\begin{prop}\label{prop:psiphom}
		Given an extension $A \xrightarrow{i} U \xrightarrow{j} S$ and an order-preserving transversal $\rho: S \to U$  of $j$, the map 
		\begin{align*}
			\psi : \cG S & \to \varsigma(A), \\ 
			g & \mapsto  [\0_g],
		\end{align*}
		defined above is a unital partial homomorphism.
	\end{prop}
	\begin{proof}
		It only suffices to show that for all $g,h \in \cG S$
		\begin{align}
			\0_{g\m} \0_g\0_h&\sim\0_{g\m}\0_{gh},\label{0_(x-inv)0_x0_y-sim-0_(x-inv)0_xy}\\
			\0_g\0_h\0_{h\m}&\sim \0_{gh}\0_{h\m},\label{0_x0_y0_(y-inv)-sim-0_xy0_(y-inv)}
		\end{align}
		since $\0_1=\id_A$ in view of \cref{TPA2}. Observe using \cref{TPA3} that 
		\begin{align*}
			\dom{(\0_g \0_h)}&=\0\m_h(\D_h \D_{g\m})=\D_{h\m}\D_{h\m g\m},\\
			\ran{(\0_g \0_h)}&=\0_g(\D_h \D_{g\m}) = \D_g\D_{gh}.
		\end{align*} 
		It follows that 
		\begin{align*}
			\dom{(\0_{g\m}\0_g\0_h)}&=\dom{(\0_{g\m}(\0_g\0_h))} = \0\m_h\0\m_g(\D_g\D_{gh})  = \D_{h\m}\D_{h\m g\m},\\
			\ran{(\0_{g\m}\0_g\0_h)}&=\ran{(\0_{g\m}(\0_g\0_h))} = \0_{g\m}(\D_g\D_{gh}) = \D_{g\m}\D_h.
		\end{align*}
		Now,
		\begin{align*}
			\dom{(\0_{g\m}\0_{gh})}&=\0\m_{gh}(\D_{gh}\D_g) = \D_{h\m g\m}\D_{h\m},\\
			\ran{(\0_{g\m}\0_{gh})}&=\0_{g\m}(\D_{gh}\D_g) = \D_h\D_{g\m},
		\end{align*} 
		so that $\0_{g\m}\0_g\0_h$ and $\0_{g\m}\0_{gh}$ have the same domain and range. For any $a \in \D_{h\m g\m}\D_{h\m}$ by \cref{TPA2,TPA4} we have
		\begin{align*}
			\0_{g\m} \0_g\0_h(a) & = (\0_{g\m} \0_g)(\0_h(a)) = w_{g\m,g} \0_h(a)w\m_{g\m,g}, \\
			\0_{g\m} \0_{gh}(a) & = w_{g\m,gh} \0_h(a) w\m_{g\m,gh},
		\end{align*}
		whence $\0_{g\m} \0_g\0_h(a)=w_{g\m,g} w\m_{g\m,gh}(\0_{g\m} \0_{gh}(a))w_{g\m,gh}w\m_{g\m,g}$, proving \cref{0_(x-inv)0_x0_y-sim-0_(x-inv)0_xy}. The proof of \cref{0_x0_y0_(y-inv)-sim-0_xy0_(y-inv)} is similar.
	\end{proof}
	
	\begin{defn}\label{def:nunkernelabs}
		Given an admissible extension $A \xrightarrow{i} U \xrightarrow{j} G$ of a semilattice of groups $A$ by a group $G$, we define the \textit{abstract kernel of this extension} as the partial homomorphism 
		\begin{align*}
			\psi : G & \to  \varsigma(A),\\
			g & \mapsto  [\0_g],
		\end{align*}
		with $\0 = \{ \0_g : \D_{g\m} \to \D_g \}$ being the action part of the twisted partial action of $G$ on $A$ induced by a refinement $A \xrightarrow{i} U \xrightarrow{\pi} S$ of the extension and an order-preserving transversal $\rho: S \to U$ of $\pi$.
		
		More generally, by an \textit{abstract kernel} (without referring to an extension) we mean a triple $(A,G,\psi)$, where $A$ is a semilattice of groups, $G$ is a group, and $\psi : G \to \varsigma(A)$ is a unital partial homomorphism.
	\end{defn}
	
	\begin{prop}\label{prop:nunequivmesmokernelabs}
		Equivalent admissible extensions of $A$ by $G$ define the same abstract kernel $\psi:G\to\varsigma(A)$.
	\end{prop}
	\begin{proof}
		By \cite[Lemma 4.1]{DK2} equivalent admissible extensions $U$ and $U'$ of $A$ by $G$ induce equivalent twisted partial actions $\Theta$ and $\Theta'$ of $G$ on $A$. Denoting by $\psi$ and $\psi'$ the corresponding abstract kernels, by \cref{equiv-Thetas} we have $\psi(g)=[\0_g]=[\0'_g]=\psi'(g)$.
	\end{proof}
	
	A general problem of extensions of semilattices of groups by groups is that of constructing all (admissible) extensions with a given abstract kernel $(A,G, \psi)$; that is, constructing all $A \xrightarrow{i} U \xrightarrow{j} G$ whose induced abstract kernel is $\psi$.
		
%
%
	\subsection{Obstruction to an extension of an abstract kernel}
	Now suppose that we are given an abstract kernel $(A,G, \psi)$. For each $g\in G$, we choose a representative 
	\begin{align}\label{0_g-in-psi(g)}
		\0_g \in \psi(g)
	\end{align}
	and denote by $\D_g$ the ideal $\ran{\0_g}$, so that $\0_g:\D_{g\m}\to\D_g$ by \cref{domain-of-0_g}. It follows from the proof of \cref{domain-of-0_g} that $\id_{\D_{h\m}}\in\psi(h\m)\psi(h)$, whence $\0_{gh}\id_{\D_{h\m}} \in \psi(gh)\psi(h\m)\psi(h)$. By \cref{propert-of-part-hom}\cref{f(g)f(h)=f(g)f(g-inv)f(gh)} we have
	\begin{align*}
	\0_g\0_h\sim\0_{gh}\id_{\D_{h\m}}.
	\end{align*} 
	Hence there exists an invertible multiplier $w_{g,h}$ of $\D_g\D_{gh}$ such that 
	\begin{align*}
	(\0_g\0_h)(a) = w_{g,h} (\0_{gh}\id_{\D_{h\m}})(a) w\m_{g,h}
	\end{align*} 
	for all $a \in \D_{h\m}\D_{h\m g\m}$. Using the notation introduced in \cref{defn-conj-by-multiplier} we may write 
	\begin{align}\label{eq:defw}
	(\0_g\0_h)(a) = \mu(w_{g,h})(\0_{gh}(a))
	\end{align}
	for all $a \in \D_{h\m}\D_{h\m g\m}$. 
	
	\begin{lem}\label{existense-of-obstr-beta}
		There exists $\beta(g,h,k)\in \U{C(\M(\D_g\D_{gh}\D_{ghk}))}$, such that 
		\begin{align}\label{eq:nun3coc}
		w_{h,k}^{\0_g}w_{g,hk} = \beta(g,h,k)w_{g,h}w_{gh,k}
		\end{align}
		for all $g,h,k\in G$.
	\end{lem}
	\begin{proof}
		By \cref{lem:tpa3} we have $\dom (\0_g \0_h \0_k)=\D_{k\m}\D_{k\m h\m}\D_{k\m h\m g\m}$. Using associativity in $\Sigma(A)$ and \cref{eq:malphaconj,eq:defw}, we compute $(\0_g\0_h\0_k)(a)$ in two different ways:
		\begin{align*}
			((\0_g\0_h)\0_k)(a) & =  \mu(w_{g,h}) (\0_{gh} \0_k (a)) \\
			& =  (\mu(w_{g,h}) \circ \mu(w_{gh,k}))(\0_{ghk}(a)) \\
			& =  \mu(w_{g,h}w_{gh,k})(\0_{ghk}(a)),\\
			(\0_g(\0_h\0_k))(a) & =  \mu(w_{h,k}^{\0_g})(\0_g \0_{hk}(a)) \\
			& =  (\mu(w_{h,k}^{\0_g}) \circ \mu(w_{g,hk})) (\0_{ghk}(a)) \\
			& =  \mu(w_{h,k}^{\0_g}w_{g,hk})(\0_{ghk}(a)).
		\end{align*}

		Since $\0_{ghk}(a)\in \D_g\D_{gh}\D_{ghk}$ (by \Cref{cor:gentpa3}), we have 
		\begin{align*}
			\mu(w_{g,h}w_{gh,k}) = \mu(w_{h,k}^{\0_g}w_{g,hk}) \mbox{ on } \D_g\D_{gh}\D_{ghk}.
		\end{align*}
		
		It remains to apply \Cref{lem:existbeta}.
	\end{proof}

	\begin{defn}
		The map $\beta$ from \cref{existense-of-obstr-beta} will be called an \textit{obstruction} to the extension of the abstract kernel $(A,G, \psi)$. 
	\end{defn}
	
	\begin{defn}\label{def:nunthetatil}
		Given an abstract kernel $(A,G, \psi)$, for each choice of representatives $\0_g\in\psi(g)$ define $\tl\0_g=\0_g|_{C(\D_{g\m})} : C(\D_{g\m}) \to C(\D_g)$.
	\end{defn}
	
	\begin{rem}
		In view of \cref{lem:trocavw} the map $\tl\0_g$ does not depend on the choice of representatives $\0_g\in\psi(g)$, and thus is uniquely determined by the abstract kernel $(A,G, \psi)$.
	\end{rem}

	\begin{prop}\label{prop:pacaocentro}
		The family $\{\tl\0_g\}_{g\in G}$ from \Cref{def:nunthetatil} is a partial action of $G$ on $C(A)$.
	\end{prop}
	\begin{proof}
		We have $\psi(1)=[\id_A]$, so $\0_1=\mu(m)$ for some $m\in\U{\M(A)}$. Then $\tl\0_1 = \mu(m)|_{C(A)}=\id_{C(A)}$, proving \cref{PA1} for $\tl\0$.
		
		Property \cref{PA2} for $\tl\0$ is \cref{cor:ideaiscentro}.
		
		For \cref{PA3} take $a \in C(\D_{h\m} \D_{h\m g\m})$ and write using \cref{eq:defw,cor:ideaiscentro} together with \cref{m-and-n-associate-with-s}\cref{ws=sw-for-s-in-C(S)}
		\begin{align*}
			(\tl\0_g \circ \tl\0_h)(a) = (\0_g \circ \0_h)(a) = w_{g,h} \0_{gh}(a) w\m_{g,h}=\0_{gh}(a)=\tl\0_{gh}(a).
		\end{align*} 
	\end{proof}
	
	
	\begin{lem}\label{exists-ext=>exists-triv-obstr}
		If an abstract kernel $(A,G, \psi)$ has an admissible extension then there exists a choice of representatives $\0_g\in\psi(g)$ such that the corresponding obstruction is the trivial $3$-cochain.
	\end{lem}
	\begin{proof}
		If $(A,G, \psi)$ is the abstract kernel of an admissible extension of $A$ by $G$, then $\psi(g)=[\0_g]$, where $(\0,w)$ is the induced twisted partial action of $G$ on $A$. Observe that \cref{TPA6} is exactly \eqref{eq:nun3coc} with trivial $\bt$. Indeed, $a' \in \D_g\D_{gh}\D_{ghk}$ if and only if $a' = \0_g(a)$ for some $a \in\D_{g\m}\D_h\D_{hk}$. Since $a' w_{h,k}^{\0_g} = \0_g(\0\m_g(a')w_{h,k}) = \0_g(aw_{h,k})$, then
		\begin{align*}
			a' w_{h,k}^{\0_g}w_{g,hk} = a' w_{g,h}w_{gh,k}
		\end{align*}
		 is equivalent to
		 \begin{align*}
		 	\0_g(aw_{h,k})w_{g,hk} = \0_g(a) w_{g,h}w_{gh,k},
		 \end{align*} 
		 which is \cref{TPA6}.
		
		
	\end{proof}

	\begin{lem}
		Let $\0_g \in \psi(g)$, $\0_g:\D_{g\m}\to\D_g$, and $w_{g,h}\in\U{\M(\D_g\D_{gh})}$ satisfying \eqref{eq:defw}. Then
		\begin{enumerate}[label=(\alph*)]
			\item\label{0_g(aw_h_k.w_hk_l)=...} for any $a \in C(\D_{g\m}\D_h \D_{hk})$:
			\begin{align}\label{eq:distheta}
			\0_g(aw_{h,k}) = \0_g(a)w_{h,k}^{\0_g};
			\end{align}
			
			\item\label{aw_k_l^(0_g0_h)=...} for any $a \in C(\D_g \D_{gh} \D_{ghk} \D_{ghkl})$:    
			\begin{align}\label{eq:compconj}
			a w_{k,l}^{\0_g\0_h} = w_{g,h} (a w_{k,l}^{\0_{gh}}) w\m_{g,h}.
			\end{align}
			
		\end{enumerate}
	\end{lem}
	\begin{proof}
		\cref{0_g(aw_h_k.w_hk_l)=...}. To prove \cref{eq:distheta}, it suffices to apply \cref{tm^af=af(af-inv(t)m)}:
		\begin{align*}
			\0_g(a)w_{h,k}^{\0_g} =  \0_g(\0\m_g (\0_g(a))w_{h,k}) =  \0_g(aw_{h,k}).
		\end{align*}
			
		\cref{aw_k_l^(0_g0_h)=...}. By \cref{prop:pacaocentro} we have
		\begin{align*}
			(\0_g\0_h)\m(a)=\0\m_h\0\m_g(a)=\tl\0\m_h\tl\0\m_g(a)=\tl\0_{h\m g\m}(a)=\0_{h\m g\m}(a).
		\end{align*} 
		Therefore, using also \cref{tm^af=af(af-inv(t)m),eq:defw} we may write
		\begin{align*}
			a w_{k,l}^{\0_g\0_h} & =  (\0_g\0_h)((\0_g\0_h)\m(a) w_{k,l}) \\
			& =  w_{g,h} \0_{gh}(\0_{h\m g\m}(a) w_{k,l}) w\m_{g,h} \\
			& =  w_{g,h} (a w_{k,l}^{\0_{gh}} ) w\m_{g,h},
		\end{align*}
		proving \cref{eq:compconj}.
		
	\end{proof}
	
	The next three lemmas are adaptations of Lemmas IV.8.4--IV.8.6 from \cite{Maclane}.
	
	\begin{lem}\label{lem:beta3coc}
		Let $(A,G,\psi)$ be an abstract kernel and $\beta$ an obstruction to $(A,G,\psi)$. Then $\beta \in Z^3(G,C(A))$, where $C(A)$ is considered as a partial $G$-module via $\tl\0$.
	\end{lem}
	\begin{proof}
		We need to show that $(\delta^3\beta)(g,h,k,l)a = a$ for all $a \in C(\D_g\D_{gh}\D_{ghk}\D_{ghkl})$, that is 
		\begin{align}\label{3-cocycle-ident-for-bt}
			\0_g(\0_{g\m}(a) \beta(h,k,l)) \beta(g,hk,l)\beta(g,h,k) = a\beta(gh,k,l)\beta(g,h,kl).
		\end{align}
		 To this end, we express 
		 $
		 	\0_g(\0_{g\m}(a)w_{k,l}^{\0_h} w_{h,kl}) w_{g,hkl}
		 $
		  in two ways using \cref{eq:mbetaalpha,eq:nun3coc,eq:distheta,eq:compconj} and \cref{m-and-n-associate-with-s}\cref{ws=sw-for-s-in-C(S)}. On the one hand
			\begin{align*}
				&\0_g(\0_{g\m}(a)\underbrace{w_{k,l}^{\0_h} w_{h,kl}}) w_{g,hkl}\\
				 &\quad= \0_g(\0_{g\m}(a)\beta(h,k,l)  w_{h,k} w_{hk,l}) w_{g,hkl} &(\text{by}\cref{eq:nun3coc})\\
				 &\quad= \0_g(\0_{g\m}(a)\beta(h,k,l)) \underbrace{w_{h,k}^{\0_g}w_{g,hk}}w\m_{g,hk} \underbrace{w_{hk,l}^{\0_g} w_{g,hkl}} &(\text{by}\cref{eq:distheta}) \\
				 &\quad= \0_g(\0_{g\m}(a)\beta(h,k,l)) \beta(g,h,k)w_{g,h}w_{gh,k}w\m_{g,hk}\beta(g,hk,l)w_{g,hk}w_{ghk,l} &(\text{by}\cref{eq:nun3coc})\\
				 &\quad= \0_g(\0_{g\m}(a)\beta(h,k,l))\beta(g,hk,l)\beta(g,h,k)w_{g,h}w_{gh,k}w_{ghk,l}.
			\end{align*}
			On the other hand
			\begin{align*}
				&\0_g(\0_{g\m}(a)w_{k,l}^{\0_h}w_{h,kl})w_{g,hkl}\\
				&\quad= a (w_{k,l}^{\0_h})^{\0_g}   w_{h,kl}^{\0_g}   w_{g,hkl} &(\text{by}\cref{eq:distheta})\\
				&\quad= aw_{k,l}^{\0_g\0_h} w_{h,kl}^{\0_g}   w_{g,hkl} &(\text{by}\cref{eq:mbetaalpha})\\
				&\quad= w_{g,h} a \underbrace{w_{k,l}^{\0_{gh}}w_{gh,kl}}w\m_{gh,kl} w\m_{g,h} \underbrace{w_{h,kl}^{\0_g}   w_{g,hkl}} &(\text{by}\cref{eq:compconj})\\
				&\quad= w_{g,h} a \beta(gh,k,l)w_{gh,k}w_{ghk,l} w\m_{gh,kl} w\m_{g,h}\beta(g,h,kl)w_{g,h}w_{gh,kl} &(\text{by}\cref{eq:nun3coc})\\
				&\quad= a \beta(gh,k,l)\beta(g,h,kl)w_{g,h}w_{gh,k}w_{ghk,l}.
			\end{align*}
		Canceling $w_{g,h}w_{gh,k}w_{ghk,l}\in\U{\M(\D_g\D_{gh}\D_{ghk}\D_{ghkl})}$, we get \cref{3-cocycle-ident-for-bt}.
	\end{proof}
	
	We shall now see how different choices of $\0$ and $w$ interfere in the obstruction to a given abstract kernel.

	\begin{lem}\label{lem:changewcohbeta}
		Another choice of $\{w_{g,h}\}_{g,h\in G}$ in \cref{eq:defw}, for the same representatives $\{\0_g\}_{g\in G}$, produces a partial $3$-cocycle cohomologous to $\beta$.
	\end{lem}
	\begin{proof}
		Let $w'_{g,h} \in \M(\D_g\D_{gh})$ be another multiplier such that
		\begin{align*}
			(\0_g\0_h)(a) = \mu(w'_{g,h})(\0_{gh}(a))
		\end{align*} 
		for all $a \in \D_{h\m}\D_{h\m g\m}$. Since any element of $\D_g\D_{gh}$ has the form $\0_{gh}(a)$ for some $a \in \D_{h\m}\D_{h\m g\m}$, we conclude that $\mu(w_{g,h}) = \mu(w'_{g,h})$ on $\D_g\D_{gh}$. By \cref{lem:existbeta} there exists $v_{g,h}\in \U{C(\M(\D_g\D_{gh}))}$ such that 
		\begin{align}\label{w'_gh=v_gh.w_gh}
			w'_{g,h} = v_{g,h}w_{g,h}
		\end{align}
		on $\D_g\D_{gh}$. Thus, $v$ is a partial $2$-cochain, i.e. $v \in C^2(G,C(A))$ in the sense of~\cite{DK3}. Let $\beta'(g,h,k)\in \U{C(\M(\D_g\D_{gh}\D_{ghk}))}$, such that 
		\begin{align}\label{bt'-obstr-to-w'}
			(w'_{h,k})^{\0_g}w'_{g,hk} =\beta'(g,h,k)w'_{g,h}w'_{gh,k},
		\end{align}
		which exists in view of \cref{existense-of-obstr-beta}. By \cref{w'_gh=v_gh.w_gh,eq:mbetaalpha,bt'-obstr-to-w'} we have
		\begin{align}\label{bt'-obstr-to-v.w}
			v_{h,k}^{\0_g}w_{h,k}^{\0_g}v_{g,hk}w_{g,hk}=(v_{h,k}w_{h,k})^{\0_g}v_{g,hk}w_{g,hk} =\beta'(g,h,k)v_{g,h}w_{g,h}v_{gh,k}w_{gh,k}.
		\end{align}
		Now, since $v_{h,k}$ is a central multiplier of $\D_h\D_{hk}$, by \cref{m-in-C(M(S))-iff-ms=sm} its restriction to the ideal $\D_{g\m}\D_h\D_{hk}$ is a central multiplier of $\D_{g\m}\D_h\D_{hk}$. As $\0_g$ is an isomorphism $\0_g : \D_{g\m} \to \D_g$, it follows that $v_{h,k}^{\0_g}$ is a central multiplier of $\0_g(\D_{g\m}\D_h\D_{hk}) = \D_g\D_{gh}\D_{ghk}$. Therefore, in view of \cref{lem:trocavw}, all the multipliers $v$ commute with all the multipliers $w$ in \cref{bt'-obstr-to-v.w} when restricted to $C(\D_g\D_{gh}\D_{ghk})$. Hence, we may write
		\begin{align*}
			v_{h,k}^{\0_g}v_{g,hk} w_{h,k}^{\0_g}w_{g,hk} = \beta'(g,h,k)v_{g,h}v_{gh,k}w_{g,h}w_{gh,k} \mbox{ on }C(\D_g\D_{gh}\D_{ghk}).
		\end{align*}
		Applying \eqref{eq:nun3coc} to $w_{h,k}^{\0_g}w_{g,hk}$, we get
		\begin{align*}
			v_{h,k}^{\0_g}v_{g,hk} \beta(g,h,k)w_{g,h}w_{gh,k} = \beta'(g,h,k)v_{g,h}v_{gh,k}w_{g,h}w_{gh,k}.
		\end{align*}
		Canceling $w_{g,h}w_{gh,k}\in\U{\D_g\D_{gh}\D_{ghk}}$, we come to
		\begin{align*}
			v_{h,k}^{\0_g}v_{g,hk} \beta(g,h,k) = \beta'(g,h,k)v_{g,h}v_{gh,k}.
		\end{align*}
		Since all the multipliers above are central and invertible on the respective ideal, we finally obtain
		\begin{align*}
			\beta'(g,h,k) = \beta(g,h,k)(\delta^2 v)(g,h,k) \mbox{ on } C(\D_g\D_{gh}\D_{ghk}),
		\end{align*}
		which proves the lemma.
	\end{proof}

	\begin{lem}\label{lem:changethetasamebeta}
		Another choice of the representatives $\{\0_g\}_{g\in G}$ may be followed by a new selection of $\{w_{g,h}\}_{g,h\in G}$ in \cref{eq:defw} to induce the same obstruction $\beta$.
	\end{lem}
	\begin{proof}
		Let $\0'_g\in\psi(g)$ be another representative of the class $\psi(g)$. Then $\0'_g:\D_{g\m}\to\D_g$ and there exists $m_g\in\U{\M(\D_g)}$
		\begin{align}\label{eq:defm}
		\0'_g(a) = \mu(m_g)(\0_g(a))
		\end{align} 
		for all $a\in\D_{g\m}$. Taking $a \in \D_{h\m}\D_{h\m g\m}$ we have
		\begin{align*}
			(\0'_g\0'_h)(a) &= m_g\0_g(m_h\0_h(a)m\m_h)m\m_g & (\text{by}\cref{eq:defm})\\
				&= m_gm_h^{\0_g}\0_g(\0_h(a)) (m_h^{\0_g})\m m\m_g & (\text{by}\cref{tm^af=af(af-inv(t)m)})\\
				&= m_gm_h^{\0_g}w_{g,h}\0_{gh}(a)w\m_{g,h} (m_h^{\0_g})\m m\m_g & (\text{by}\cref{eq:defw})\\
				&= m_gm_h^{\0_g}w_{g,h}m\m_{gh}\0'_{gh}(a)m_{gh}w\m_{g,h} (m_h^{\0_g})\m m\m_g & (\text{by}\cref{eq:defm})\\
				&= \mu(w'_{g,h})(\0'_{gh}(a)),
		\end{align*}
		where 
		\begin{align}\label{w'_gh=m_gh^0_gw_ghm^(-1)_gh}
			w'_{g,h}:=m_gm_h^{\0_g}w_{g,h}m\m_{gh}.
		\end{align}
		Observe using \cref{lem:mgthetatroca} that
		\begin{align}\label{eq:w'theta'}
			w'_{g,h} m_{gh}  = m_g m_h^{\0_g} w_{g,h} = m_h^{\0'_g} m_g w_{g,h}.
		\end{align}
		Let $\beta'$ be the obstruction induced by the pair $(\0',w')$. We shall show that $\bt'$ in fact coincides with $\bt$, which will prove the lemma. To this end, we compute
		\begin{align*}
			(w'_{h,k})^{\0'_g} \underbrace{w'_{g,hk} m_{ghk}} &= (w'_{h,k})^{\0'_g} m_{hk}^{\0'_g}   m_g   w_{g,hk} & (\text{by}\cref{eq:w'theta'})\\
			&= (\underbrace{w'_{h,k} m_{hk}})^{\0'_g}m_gw_{g,hk} &(\text{by \cref{m->m^af-iso}})\\
			&= ( m_k^{\0'_h} m_h w_{h,k})^{\0'_g} m_g w_{g,hk} &(\text{by}\cref{eq:w'theta'})\\
			&= (m_k^{\0'_h})^{\0'_g} m_h^{\0'_g} \underbrace{w_{h,k}^{\0'_g} m_g}w_{g,hk} &(\text{by \cref{m->m^af-iso}})\\
			&=  m_k^{\0'_g\0'_h} m_h^{\0'_g} m_g   \underbrace{w_{h,k}^{\0_g} w_{g,hk}} & (\text{by \cref{lem:mgthetatroca,lem:malpha}})\\
			&= m_k^{\0'_g\0'_h} m_h^{\0'_g} m_g \beta(g,h,k) w_{g,h} w_{gh,k} &(\text{by \cref{eq:nun3coc}})\\
			&= \beta(g,h,k)m_k^{\0'_g\0'_h} m_h^{\0'_g} m_g w_{g,h} w_{gh,k}. &(\text{since }\beta(g,h,k)\text{ is central})
		\end{align*}
		Then, for any $a \in C(\D_g\D_{gh}\D_{ghk})$, we have
		\begin{align*}
			a(w'_{h,k})^{\0'_g} w'_{g,hk} m_{ghk} &= \beta(g,h,k)am_k^{\0'_g\0'_h} m_h^{\0'_g} m_g w_{g,h} w_{gh,k} \\
			&= \beta(g,h,k) w'_{g,h} a m_k^{\0'_{gh}} (w'_{g,h})\m \underbrace{m_h^{\0'_g} m_g w_{g,h}} w_{gh,k} & (\text{by \cref{eq:compconj}})\\
			&= \beta(g,h,k) w'_{g,h} a m_k^{\0'_{gh}} \underbrace{(w'_{g,h})\m w'_{g,h}} m_{gh} w_{gh,k} & (\text{by \cref{eq:w'theta'}})\\
			&= \beta(g,h,k) w'_{g,h} a \underbrace{m_k^{\0'_{gh}} m_{gh} w_{gh,k}}\\
			&= \beta(g,h,k) w'_{g,h} a w'_{gh,k} m_{ghk} & (\text{by \cref{eq:w'theta'}})\\
			&= a\beta(g,h,k) w'_{g,h} w'_{gh,k} m_{ghk}. & (\text{since }a\text{ is central})
		\end{align*}
		Canceling the invertible multiplier $m_{ghk}$, we conclude that
		\begin{align*}
			(w'_{h,k})^{\0'_g} w'_{g,hk}=\beta(g,h,k) w'_{g,h} w'_{gh,k}
		\end{align*}
		on $C(\D_g\D_{gh}\D_{ghk})$. In view of \cref{eq:nun3coc} the latter means that $\beta'(g,h,k)=\beta(g,h,k)$.
	\end{proof}
	
	\begin{lem}\label{norm-v-in-dl^2v}
		Let $(A,\0)$ be a partial $G$-module in the sense of~\cite{DK3} and $v\in C^2(G,A)$. Then there exists $v'\in C^2(G,A)$ such that $\delta^2v=\delta^2v'$ and $v'_{1,1}=\id_A$. 
	\end{lem}
	\begin{proof}
		Define $u\in C^1(G,A)$ by
		\begin{align*}
			u_xa=au_x=av\m_{1,1}=v\m_{1,1}a,
		\end{align*}
		where $a\in\D_x$. Then, for any $a\in \D_x\D_{xy}$, we have
		\begin{align*}
			a(\delta^1u)_{x,y}=\0_x(\0_{x\m}(a)u_y)u\m_{xy}u_x=\0_x(\0_{x\m}(a)v\m_{1,1})v_{1,1}v\m_{1,1}=\0_x(\0_{x\m}(a)v\m_{1,1}).
		\end{align*}
		Now, set $v':=v\cdot \dl^1 u\in C^2(G,A)$. We have $\dl^2v'=\dl^2v\cdot\dl^2\dl^1u=\dl^2v$ and $av'_{1,1}=v_{1,1}\0_1(\0_1(a)v\m_{1,1})=a$, as $v_{1,1}$ is a central multiplier of $A$.
	\end{proof}

	We can summarize our results in the following.
	\begin{thrm}\label{theo:main1}
		Given an abstract kernel $(A,G,\psi)$, the center $C(A)$ can be uniquely regarded as a partial $G$-module via $\tl\0 = \{\tl\0_g : C(\D_{g\m}) \to C(\D_g)\}_{g \in G}$, where $\0_g \in \psi(g)$ and $\tl\0_g=\0_g|_{C(\D_{g\m})}$. Taking the cohomology class of an obstruction $\bt$ to $(A,G,\psi)$, we have a well-defined element $\operatorname{Obs}(A,G,\psi) \in H^3(G,C(A))$. The abstract kernel $(A,G,\psi)$ has an admissible extension if and only if $\operatorname{Obs}(A,G,\psi)$ is trivial.
	\end{thrm}
	\begin{proof}
		The \textit{``only if''} part is explained by \cref{prop:pacaocentro,exists-ext=>exists-triv-obstr,lem:beta3coc,lem:changewcohbeta,lem:changethetasamebeta}.
		
		It remains to prove that the trivial obstruction yields the existence of an extension for $(A,G,\psi)$. Indeed, let $(\0,w)$ be such that $\beta = \delta^2 v$. In view of \cref{norm-v-in-dl^2v} we may assume that $v_{1,1}=\id_A$. Moreover, since the class of $\bt$ does not depend on the choice of $(\0,w)$ satisfying \cref{eq:defw,0_g-in-psi(g)}, we may also take $\0_1=\id_A$ and $w_{1,1}=\id_A$. 
	
		By the proof \Cref{lem:changewcohbeta}, if we define $w'_{g,h}=v\m_{g,h}w_{g,h}$, then the corresponding $\bt'$ satisfies 
		$$
		\bt'(g,h,k)=(\delta^2 v)(g,h,k)(\delta^2 v\m)(g,h,k)=\id_{C(\D_g\D_{gh}\D_{ghk})}.
		$$
		It follows from \cref{eq:nun3coc} that $w'$ satisfies \cref{TPA6}, and moreover $w'_{1,1}=v\m_{1,1}w_{1,1}=\id_A$. Writing \cref{TPA6} for the triples $(1,1,g)$ and $(g,1,1)$ we obtain $w'_{1,g}=w'_{g,1}=\id_{\D_g}$. Thus, $(\0,w')$ is a twisted partial action of $G$ on $A$. It induces the admissible extension $A\overset{i}{\to} A*_{(\0,w')}G\overset{j}{\to} G$ constructed in \cite[Proposition 5.15]{DK2}. Moreover, by \cite[Lemma 4.4]{DK3} there is a refinement $A\overset{i}{\to} A*_{(\0,w')}G\overset{\pi}{\to}E(A)*_{(\0,w')}G\overset{\kappa}{\to} G$ and an order-preserving transversal $\rho$ of $\pi$, such that the induced twisted partial action of $G$ on $A$ coincides with $(\0,w')$. Then the abstract kernel of this extension coincides with $\psi$.
	\end{proof}
	
	
	\subsection{The description of all the extensions of an abstract kernel}\label{sec-descr-ext}
	
	\begin{defn}\label{defn-of-v'}
		Let $\Theta=(\0,w)$ be a partial action of $G$ on $A$, $\tl\0$ the restriction of $\0$ to $C(A)$ and $v \in C^2(G,C(A))$ with respect to $\tl\0$. We define 
		\begin{align*}
			v'=\{v'_{g,h}\mid g,h\in G\},
		\end{align*}
		where $v'_{g,h}$ is the extension of $v_{g,h}$ to a central multiplier of $\D_g\D_{gh}$ as in \cref{prop:defv'}. 
	\end{defn}

	\begin{lem}\label{lem:v'tpa6}
		Let $\Theta=(\0,w)$ be a twisted partial action of $G$ on $A$ and $v \in Z^2(G,C(A))$. Then $(\0,v')$ satisfies \cref{TPA6}.
	\end{lem}
	\begin{proof}
		We need to prove that
		\begin{align}\label{TPA6-for-v'}
			(v'_{h,k})^{\0_g} v'_{g,hk} = v'_{g,h} v'_{gh,k}
		\end{align}
		 on $\D_g\D_{gh}\D_{ghk}$. An arbitrary element of $\D_g\D_{gh}\D_{ghk}$ has the form $\0_g(a)$, where $a \in \D_{g\m}\D_h\D_{hk}$. Applying the left-hand side of \cref{TPA6-for-v'} to $\0_g(a)$ and using \cref{tm^af=af(af-inv(t)m),am'=a(a-inv-a)m}, we have
		\begin{align*}
			\0_g(a) (v'_{h,k})^{\0_g} v'_{g,hk} &= \0_g(\0\m_g(\0_g(a)) v'_{h,k}) v'_{g,hk} \\
			&= \0_g(av'_{h,k}) v'_{g,hk} \\
			&= \0_g(a(a\m a)v_{h,k}) v'_{g,hk}  \\
			&= \0_g(a(a\m a)v_{h,k})(\0_g(a\m(a\m a)v\m_{h,k})\0_g(a(a\m a)v_{h,k})) v_{g,hk}.
		\end{align*}
		Now, $\0_g(a(a\m a)v_{h,k})=\0_g(a)\tl\0_g((a\m a)v_{h,k})$, as $a\m a\in C(A)$ and $v_{h,k}\in\M(C(A))$. Moreover,	since central elements commute with multipliers, we have
		\begin{align*}
			\0_g(a\m(a\m a)v\m_{h,k})\0_g(a(a\m a)v_{h,k})&=\tl\0_g(a\m(a\m a)v\m_{h,k}a(a\m a)v_{h,k})\\
			&=\tl\0_g(a\m a(a\m a)v\m_{h,k}(a\m a)v_{h,k})\\
			&=\tl\0_g((a\m a)v\m_{h,k}v_{h,k})\\
			&=\tl\0_g(a\m a).
		\end{align*}
		Hence,
		\begin{align*}
			\0_g(a) (v'_{h,k})^{\0_g} v'_{g,hk} &=\0_g(a(a\m a)v_{h,k})\tl\0_g(a\m a)v_{g,hk}\\
			&=\0_g(a)\tl\0_g((a\m a)v_{h,k})v_{g,hk}\\
			&= \0_g(a)\tl\0_g(a\m a)v_{h,k}^{\0_g} v_{g,hk}.
		\end{align*}
		Similarly,
		\begin{align*}
			\0_g(a)v'_{g,h}v'_{gh,k}=\0_g(a)\tl\0_g(a\m a) v_{g,h} v_{gh,k}.
		\end{align*}
		It remains to use the partial $2$-cocycle identity for $v$.
		
	\end{proof}
	
	\begin{defn}
		Let $\Theta=(\0,w)$ be a twisted partial action of $G$ on $A$ and $v \in C^2(G,C(A))$. Define 
		\begin{align}\label{defn-of-vTheta}
		v\Theta  = (\0, v'w),
		\end{align}
		 where $v'$ is given by \cref{defn-of-v'}.
	\end{defn}

	Following \cite{DK}, a partial $2$-cocycle $f$ of $G$ with values in a (not necessarily unital) partial $G$-module $A$ will be called \textit{normalized}, if $f(1,1)=\id_A$. This readily implies that $f(1,g)=f(g,1)=\id_{\D_g}$. As in \cite[Remark 2.6]{DK}, one can prove that any partial $2$-cocycle is cohomologous to a normalized one. The subgroup of normalized partial $2$-cocycles will be denoted by $NZ^2(G,A)$.

	\begin{lem}\label{prop:vThetatpa}
		Let $\Theta = (\0, w)$ be a twisted partial action of $G$ on $A$ and $v \in C^2(G,C(A))$. Then $v\Theta$ is also a twisted partial action of $G$ on $A$ if and only if $v\in NZ^2(G,C(A))$.
	\end{lem}
	\begin{proof}
		Clearly, $v\Theta$ satisfies \cref{TPA1,TPA2,TPA3} if and only if $\Theta$ does, so we need to prove that $v\Theta$ satisfies \cref{TPA4,TPA5,TPA6} if and only if $v\in NZ^2(G,C(A))$.
		 
		Let $v\in NZ^2(G,C(A))$. Condition \cref{TPA4} for $v\Theta$ follows from \cref{m-in-C(M(S))-iff-ms=sm}:
		\begin{align*}
			(v'_{g,h}w_{g,h}) \0_{gh}(a) (v'_{g,h}w_{g,h})\m & =  v'_{g,h}w_{g,h} \0_{gh}(a) w\m_{g,h}(v'_{g,h})\m \\
			& =  w_{g,h} \0_{gh}(a) w\m_{g,h} v'_{g,h}(v'_{g,h})\m \\
			& =  w_{g,h} \0_{gh}(a) w\m_{g,h}\\
			& =  \0_g \circ \0_h(a),\  a\in\D_{h\m}\D_{h\m g\m}.
		\end{align*}
		
		For \cref{TPA5} we use the fact that $v$ is normalized: 
		\begin{align*}
			a(v'w)_{1,g} = av'_{1,g}=a(a\m a)v_{1,g} = a(a\m a) = a,\  a\in\D_g.
		\end{align*}
		
		Finally, for \cref{TPA6} observe that 
		\begin{align}\label{a(v'w)_gh=(a-inv a)v'_gh.aw_gh}
			a(v'w)_{g,h}=(a(a\m a)v_{g,h})w_{g,h}=(a\m a)v_{g,h}\cdot aw_{g,h}=(a\m a)v'_{g,h}\cdot aw_{g,h},
		\end{align} 
		where $a\in\D_g\D_{gh}$. It follows that for any $a\in\D_{g\m}\D_h\D_{hk}$ we have
		\begin{align*}
			\0_g(a(v'w)_{h,k})(v'w)_{g,hk}&=\0_g((a\m a)v'_{h,k}\cdot aw_{h,k})v'_{g,hk}w_{g,hk} & \text{(by \cref{a(v'w)_gh=(a-inv a)v'_gh.aw_gh})}\\
			&=\0_g((a\m a)v'_{h,k})v'_{g,hk}\cdot\0_g(aw_{h,k})w_{g,hk} & \text{($v'_{g,hk}$ is central)}\\
			&=\0_g(a\m a)v'_{g,h} v'_{gh,k}\cdot\0_g(a)w_{g,h} w_{gh,k} & \text{(by \cref{lem:v'tpa6})}\\
			&=(\0_g(a\m a)v'_{g,h} \cdot\0_g(a)w_{g,h}) v'_{gh,k}w_{gh,k} & \text{($v'_{g,hk}$ is central)}\\
			&=\0_g(a)(v'w)_{g,h} (v'w)_{gh,k}. & \text{(by \cref{lem:v'tpa6})}
		\end{align*}
		
		Conversely, suppose that \cref{TPA4,TPA5,TPA6} hold for $v\Theta$. Condition \cref{TPA5} readily implies that $v_{1,1}=\id_A$. It remains to prove that $(\delta^2v)(g,h,k)$ is a trivial multiplier. As it was observed in the proof of \cref{exists-ext=>exists-triv-obstr}, condition \cref{TPA6} for $\Theta$ is equivalent to
		\begin{align*}
			w_{h,k}^{\0_g}w_{g,hk} = w_{g,h}w_{gh,k} \mbox{ on }\D_g\D_{gh}\D_{ghk}.
		\end{align*}
		Applying this observation to $v\Theta$, we see that
		\begin{align*}
			(v_{h,k}w_{h,k})^{\0_g}v_{g,hk}w_{g,hk} = v_{g,h}w_{g,h}v_{gh,k}w_{gh,k} \mbox{ on }C(\D_g\D_{gh}\D_{ghk}).
		\end{align*}
		Since the values of $v$ are central multipliers and $(v_{h,k}w_{h,k})^{\0_g}=v_{h,k}^{\0_g}w_{h,k}^{\0_g}$ by \cref{m->m^af-iso}, we conclude that
		\begin{align*}
			v_{h,k}^{\0_g}v_{g,hk} = v_{g,h}v_{gh,k} \mbox{ on }C(\D_g\D_{gh}\D_{ghk}),
		\end{align*}
		the latter being equivalent to $(\delta^2v)(g,h,k) = \id_{C(\D_g\D_{gh}\D_{ghk})}$.
	\end{proof}
	
	\begin{lem}\label{Theta-equiv-vTheta}
		Let $\Theta = (\0, w)$ be a twisted partial action of $G$ on $A$ and $v \in NZ^2(G,C(A))$ with respect to $\tl\0$. Then $v\Theta$ is equivalent to $\Theta$ if and only if $v\in B^2(G,C(A))$.
	\end{lem}
	\begin{proof}
		Using \cref{defn-equiv-tw-pact}, we see that $v\Theta$ is equivalent to $\Theta$ if and only if there exists $\ve_g\in\U{\M(\D_g)}$, such that
		\begin{align}
			\0_g(a)&=\ve_g\0_g(a)\ve\m_g\mbox{ for all }a\in\D_{g\m},\label{0_g=ve_g.0_g.ve^(-1)_g}\\
			\0_g(a)(\0_g(a\m a)v_{g,h})w_{g,h}\ve_{gh}&=\ve_g\0_g(a\ve_h)w_{g,h}\mbox{ for all }a\in\D_{g\m}\D_h.\label{w-equiv-vw}
		\end{align}
		By \cref{m-in-C(M(S))-iff-ms=sm} condition \cref{0_g=ve_g.0_g.ve^(-1)_g} is equivalent to the fact that $\ve_g\in\U{C(\M(\D_g))}$. It follows that $w_{g,h}\ve_{gh}=\ve_{gh}w_{g,h}$ on $\D_g\D_{gh}$, so that we may cancel $w_{g,h}$ in \cref{w-equiv-vw} to get
		\begin{align*}
			\0_g(a)(\0_g(a\m a)v_{g,h})\ve_{gh}=\ve_g\0_g(a\ve_h)\mbox{ for all }a\in\D_{g\m}\D_h.
		\end{align*}
		This is equivalent to
		\begin{align}\label{a(a^(-1)a)v_g_h=0_g(0^(-1)_g(a)ve_h)ve^(-1)_gh.ve_g}
		a(a\m a)v_{g,h}=\0_g(\0\m_g(a)\ve_h)\ve\m_{gh}\ve_g\mbox{ for all }a\in\D_g\D_{gh}.
		\end{align}
		If $a\in C(\D_g\D_{gh})$, then we obtain 
		\begin{align}\label{av_g_h=a(dl^1.ve)(g_h)}
			av_{g,h}=a(\dl^1\ve)(g,h),
		\end{align}
		i.e. $v\in B^2(G,C(A))$.
		
		Conversely, if $v=\dl^1\ve$ and $a\in\D_g$, then writing \cref{av_g_h=a(dl^1.ve)(g_h)} for $a\m a\in C(\D_g)$ and multiplying by $a$ on the left we obtain \cref{a(a^(-1)a)v_g_h=0_g(0^(-1)_g(a)ve_h)ve^(-1)_gh.ve_g}, which gives \cref{w-equiv-vw}.
	\end{proof}

	\begin{lem}\label{(0_w)-equiv-(0'_w')}
		Let $\Theta=(\0,w)$ and $\Theta'=(\0',w')$ be two twisted partial actions of $G$ on $A$ with the same domains $\{\D_g\}_{g\in G}$, such that \cref{eq:defm,w'_gh=m_gh^0_gw_ghm^(-1)_gh} hold for some $m_g\in\U{\M(\D_g)}$. Then $\Theta$ and $\Theta'$ are equivalent.
	\end{lem}
	\begin{proof}
		It suffices to show that \cref{w'_gh=m_gh^0_gw_ghm^(-1)_gh} implies
		\begin{align*}
			\0'_g(a)w'_{g,h}m_{gh}=m_g\0_g(am_h)w_{g,h} \mbox{ for all }a\in\D_{g\m}\D_h.
		\end{align*}
		Multiplying both sides of \cref{w'_gh=m_gh^0_gw_ghm^(-1)_gh} by $m_{gh}$ on the right and applying them to $\0'_g(a)$, $a\in\D_{g\m}\D_h$, on the left, we obtain
		\begin{align*}
			\0'_g(a)w'_{g,h}m_{gh}&=\0'_g(a)m_h^{\0'_g} m_g w_{g,h} & \mbox{(by \cref{eq:w'theta'})}\\
			&= \0'_g(am_h)m_gw_{g,h} & \mbox{(by \cref{tm^af=af(af-inv(t)m)})}\\ 
			&= m_g\0_g(am_h)w_{g,h}. & \mbox{(by \cref{eq:defm})}
		\end{align*} 
	\end{proof}

	\begin{thrm}\label{theo:main2}
		If an abstract kernel $(A,G,\psi)$ has an admissible extension, then the set of equivalence classes of admissible extensions of $(A,G,\psi)$ is in a one-to-one correspondence with the set $H^2(G, C(A))$.
	\end{thrm}
	\begin{proof}
		We shall adapt the proof of~\cite[Theorem IV.8.8]{Maclane}, namely, we shall prove that $H^2(G, C(A))$ acts transitively and freely on the set of equivalence classes of admissible extensions of $A$ by $G$.
		
		By \cite[Theorem 6.12]{DK2} it is enough to consider the classes of admissible extensions of the form 
		\begin{align}\label{A->A*G->G}
			A \to A*_\Theta G \to G,
		\end{align}
		where $\Theta = (\0, w)$ is a twisted partial action of $G$ on $A$.
			
		To define the action of $H^2(G,C(A))$, we represent an element of $H^2(G,C(A))$ as the class $[v]$ of $v \in NZ^2(G,C(A))$. Then we define the result of the action of $[v]$ on the class of the extension \cref{A->A*G->G} as being the class of
		\begin{align*}
			A \to A*_{v\Theta} G \to G,
		\end{align*}
		where $v\Theta$ is given by \Cref{defn-of-vTheta}. In view of \Cref{prop:vThetatpa} this is well defined.
			
		Let us prove that the action is transitive. Given two twisted partial actions $\Theta=(\0, w)$ and $\Theta'=(\0', w')$ of $G$ on $A$, such that $\psi$ is the abstract kernel of both $A*_\Theta G$ and $A*_{\Theta'} G$ seen as extensions of $A$ by $G$, we have $[\0_g]=\psi(g)=[\0'_g]$ in $\varsigma(A)$. Then there is $m_g\in\U{\M(\D_g)}$, such that \cref{eq:defm} holds. Let $w''_{g,h}$ be defined by the right hand side of \cref{w'_gh=m_gh^0_gw_ghm^(-1)_gh}. By \cref{lem:changethetasamebeta} the pair $(\0',w'')$ determines the same (trivial) obstruction as $\Theta$, i.e. it is a twisted partial action of $G$ on $A$. Moreover, $(\0',w'')$ is equivalent to $\Theta$ by \cref{(0_w)-equiv-(0'_w')}. On the other hand, by the proof of \cref{lem:changethetasamebeta} we see that
		\begin{align*}
			\mu(w'_{g,h})(\0'_{gh}(a))=(\0'_g\0'_h)(a)=\mu(w''_{g,h})(\0'_{gh}(a)) \mbox{ for all }a\in\D_{g\m}\D_{h\m g\m}.
		\end{align*}
		Hence, by \cref{lem:existbeta} there is $v_{g,h}\in\U{C(\M(\D_g\D_{gh}))}$, such that $w''_{g,h}=v_{g,h}w'_{g,h}$. It follows that $\Theta$ is equivalent to $v\Theta'$, where $v \in NZ^2(G,C(A))$ thanks to \cref{prop:vThetatpa}. Thus, $A*_\Theta G$ and $A*_{v\Theta'} G$ are equivalent as extensions of $A$ by $G$ in view of \cite[Lemma 4.7]{DK3}.
		
		We now prove that the action is free. Let $\Theta = (\0,w)$ be a twisted partial action of $G$ on $A$, such that $A*_\Theta G$ and $A*_{v\Theta} G$ are equivalent as extensions of $A$ by $G$. Then $\Theta$ is equivalent to $v\Theta$ by \cite[Lemmas 4.1 and 4.4]{DK3}. Thus, $v \in B^2(G,C(A))$ thanks to \cref{Theta-equiv-vTheta}.
	\end{proof}

	\section*{Acknowledgments}
The first  author was partially supported by Conselho Nacional de Desenvolvimento Cient\'ifico e Tecnol\'ogico --- CNPq of Brazil (Proc. 307873/2017-0) and  Funda\c{c}\~ao de Amparo \`a Pesquisa do Estado de S\~ao Paulo --- FAPESP of Brazil (Proc. 2015/09162-9). The second author was partially supported by  CNPq of Brazil (Proc. 404649/2018-1) and  Funda\c{c}\~ao para a Ci\^{e}ncia e a Tecnologia (Portuguese Foundation for Science and Technology) through the project PTDC/MAT-PUR/31174/2017. The third author was financed by 
Coordena\c{c}\~ao de Aperfei\c{c}oamento de Pessoal de N\'ivel Superior - Brasil (CAPES) - Finance Code 001. We thank the anonymous referee for the careful reading of our paper and remarks that helped us to correct several misprints and improve the proof of \cref{prop:vThetatpa}.

	\bibliography{bibl-pact}{}

\begin{thebibliography}{10}

\bibitem{BaMoTe}
{\sc Batista, E., Mortari, A. D.~M., and Teixeira, M.~M.}
\newblock Cohomology for partial actions of {H}opf algebras.
\newblock {\em J. Algebra 528\/} (2019), 339--380.

\bibitem{Clifford-Preston-2}
{\sc Clifford, A., and Preston, G.}
\newblock {\em {The algebraic theory of semigroups}}, vol.~2 of {\em {Math.
  Surveys and Monographs 7}}.
\newblock Amer. Math. Soc., Providence, Rhode Island, 1967.

\bibitem{D3}
{\sc Dokuchaev, M.}
\newblock Recent developments around partial actions.
\newblock {\em S\~ao Paulo J. Math. Sci. 13}, 1 (2019), 195--247.

\bibitem{DE}
{\sc Dokuchaev, M., and Exel, R.}
\newblock {Associativity of crossed products by partial actions, enveloping
  actions and partial representations}.
\newblock {\em Trans. Amer. Math. Soc. 357}, 5 (2005), 1931--1952.

\bibitem{DES1}
{\sc Dokuchaev, M., Exel, R., and Sim{\'o}n, J.~J.}
\newblock {Crossed products by twisted partial actions and graded algebras}.
\newblock {\em J. Algebra 320}, 8 (2008), 3278--3310.

\bibitem{DES2}
{\sc Dokuchaev, M., Exel, R., and Sim{\'o}n, J.~J.}
\newblock {Globalization of twisted partial actions}.
\newblock {\em Trans. Amer. Math. Soc. 362}, 8 (2010), 4137--4160.

\bibitem{DK}
{\sc Dokuchaev, M., and Khrypchenko, M.}
\newblock {Partial cohomology of groups}.
\newblock {\em J. Algebra 427\/} (2015), 142--182.

\bibitem{DK2}
{\sc Dokuchaev, M., and Khrypchenko, M.}
\newblock Twisted partial actions and extensions of semilattices of groups by
  groups.
\newblock {\em Int. J. Algebra Comput. 27}, 7 (2017), 887--933.

\bibitem{DK3}
{\sc Dokuchaev, M., and Khrypchenko, M.}
\newblock Partial cohomology of groups and extensions of semilattices of
  abelian groups.
\newblock {\em J. Pure Appl. Algebra 222\/} (2018), 2897--2930.

\bibitem{DN}
{\sc Dokuchaev, M., and Novikov, B.}
\newblock {Partial projective representations and partial actions}.
\newblock {\em J. Pure Appl. Algebra 214}, 3 (2010), 251--268.

\bibitem{DoPaPi}
{\sc Dokuchaev, M., Paques, A., and Pinedo, H.}
\newblock Partial {G}alois cohomology and related homomorphisms.
\newblock {\em Quarterly J. Math. 70}, 2 (2019), 737--766.

\bibitem{DoPaPiRo}
{\sc Dokuchaev, M., Paques, A., Pinedo, H., and Rocha, I.}
\newblock Partial generalized crossed products and a seven-term exact sequence.
\newblock {\em arXiv:1908.05820\/}.

\bibitem{DoSa}
{\sc Dokuchaev, M., and Sambonet, N.}
\newblock Schur's theory for partial projective representations.
\newblock {\em Israel J. Math. 232}, 1 (2019), 373--399.

\bibitem{E0}
{\sc Exel, R.}
\newblock {Twisted partial actions: a classification of regular
  {$C^*$}-algebraic bundles}.
\newblock {\em Proc. London Math. Soc. 74}, 3 (1997), 417--443.

\bibitem{KennedySchafhauser}
{\sc Kennedy, M., and Schafhauser, C.}
\newblock Noncommutative boundaries and the ideal structure of reduced crossed
  products.
\newblock {\em Duke Math. J. 168}, 17 (2019), 3215--3260.

\bibitem{Lausch}
{\sc Lausch, H.}
\newblock {Cohomology of inverse semigroups}.
\newblock {\em J. Algebra 35\/} (1975), 273--303.

\bibitem{Lawson}
{\sc Lawson, M.~V.}
\newblock {\em {Inverse semigroups. {T}he theory of partial symmetries}}.
\newblock World Scientific, Singapore-New Jersey-London-Hong Kong, 1998.

\bibitem{Maclane}
{\sc Maclane, S.}
\newblock {\em {Homology}}.
\newblock Springer-Verlag, Berlin-Guttingen-Heidelberg, 1963.

\bibitem{Petrich}
{\sc Petrich, M.}
\newblock {\em {Inverse semigroups}}.
\newblock {Pure and Applied Mathematics (New York)}. John Wiley \& Sons, Inc.,
  New York, 1984.
\newblock A Wiley-Interscience Publication.

\end{thebibliography}
	\bibliographystyle{acm}
	
\end{document}